\documentclass[letterpaper, 10 pt, conference]{ieeeconf}  

\IEEEoverridecommandlockouts                              
\usepackage{amsfonts}
\usepackage{amsmath}
\usepackage{amsthm}
\usepackage{amssymb}
\usepackage{tikz}
\usetikzlibrary{arrows, decorations.pathmorphing, backgrounds, positioning, fit}
\usetikzlibrary{patterns}
\usepackage{txfonts}
\usepackage{algpseudocode}
\usepackage{algorithm}
\usepackage [font=small] {caption}
\usepackage{cite}

\newtheorem{mydef}{Definition}
\newtheorem{mylem}{Lemma}
\newtheorem{mytheo}{Theorem}
\newtheorem{myrem}{Remark}

\newtheorem{mycol}{Corollary}

\title{\LARGE \bf
Robust Network Reconstruction in Polynomial Time
}

\author{David Hayden, Ye Yuan and Jorge Gon\c{c}alves
\thanks{Department of Engineering, University of Cambridge, UK.
{\tt\small \{dph34, yy311, jmg77\}@cam.ac.uk}}
}

\begin{document}

\maketitle
\thispagestyle{empty}
\pagestyle{empty}

\begin{abstract}
This paper presents an efficient algorithm for robust network reconstruction of Linear Time-Invariant (LTI) systems in the presence of noise, estimation errors and unmodelled nonlinearities. The method here builds on previous work \cite{rQP} on robust reconstruction to provide a practical implementation with polynomial computational complexity. Following the same experimental protocol, the algorithm obtains a set of structurally-related candidate solutions spanning every level of sparsity. We prove the existence of a magnitude bound on the noise, which if satisfied, guarantees that one of these structures is the correct solution. A problem-specific model-selection procedure then selects a single solution from this set and provides a measure of confidence in that solution. Extensive simulations quantify the expected performance for different levels of noise and show that significantly more noise can be tolerated in comparison to the original method.
\end{abstract}

\section{Introduction}

An active topic of dynamical systems research and an ubiquitous problem is that of \emph{network reconstruction}, or inferring information about the structural and dynamical properties of a networked system \cite{rQP, corenet, pappas, csens, dbn, aracne, bool}. This problem is motivated by a diverse range of fields where an unknown system can be described by interconnecting subsystems acting between accessible system states. By perturbing the system and measuring these states, we seek to understand the nature and dynamics of their interactions. In general, this is an underdetermined problem and is further complicated by the fact that not all states may be available for measurement and that there may also be additional, unknown states. Here we focus on noisy LTI systems and seek to obtain reliable structural and dynamic information involving the measured states, whilst leaving the hidden states unrestricted.

One prominent example from Systems Biology is the identification of gene regulatory networks, which describe the interactions between genes via biochemical mechanisms. Despite the stochastic and nonlinear nature of biological systems, a linear-systems description provides an appealing and tractable option. Measurements of the expression levels of individual genes are routinely available (for example from microarray experiments) and a number of these are taken as system states. Additional \lq{}hidden\rq{} states are required to accurately model the effect of protein and other metabolic interactions, any process with higher than first-order dynamics and any important genes that are not selected as system states. In such applications, the assumption of full-state measurement may well lead to incorrect conclusions.

Nevertheless, there are a number of published methods to obtain a representative linear model from full state measurements, for example \cite{corenet} using Linear Matrix Inequalities, \cite{pappas} using a 1-norm residual and \cite{csens} using compressive sensing. There are many other approaches to network reconstruction, for example Bayesian \cite{dbn}, information theoretic \cite{aracne} and Boolean \cite{bool}, in all of which the solution is biased, most commonly towards sparsity, to compensate for the lack of information. The vice of all these methods is that if the true solution is not sparse, a sparse solution will be obtained anyway. 

Dynamical structure functions were introduced in \cite{QP2007} as a means of representing the structure and dynamics of an LTI system at a resolution consistent with the number of measured states. Exactly how much additional information is required to reconstruct a network from the system transfer matrix could then be quantified. In particular, without extra information, it is possible to obtain a solution that justifies any prior assumption. Applications of dynamical structure functions include multi-agent systems \cite{ye1,ye2} in addition to network reconstruction \cite{QP2008}, where necessary and sufficient conditions were given for exact reconstruction from the transfer matrix.

In \cite{rQP}, this approach was made robust to uncertainty in the transfer matrix by obtaining dynamical structure functions that are closest, in some sense, to the data. Since the dynamical structure function is no longer unique, it is also necessary to estimate Boolean network structure, that is, an unweighted directed graph of causal connections. The approach taken was to calculate the optimum dynamical structure function for every possible Boolean network for a given number of measured states, $p$, then use a model selection technique to select the best estimate. This approach suffers chiefly from a high computational complexity of $O(2^p)$, which limits its use to relatively small networks.

Here we present an extension of \cite{rQP} that does not require all of the Boolean network structures to be considered. In fact, we obtain a set of $p^2-p+1$ candidate structures by judiciously removing links from the fully-connected structure, with complexity in the number of structures that must be considered of $O(p^3)$. This set contains one structure for every level of sparsity and hence there is no prior bias towards sparser solutions. 

Section \ref{DSFsec} defines dynamical structure functions and states necessary previous results. The main result of a polynomial-time reconstruction algorithm is then presented in Section \ref{mainsec}. Section \ref{AICsec} introduces a new model selection procedure. Section \ref{simsec} then compares the performance of two variants of the method introduced here with that of \cite{rQP} in extensive random simulations. Conclusions and an outline of future work are given in Section \ref{concsec}.

\section{Dynamical Structure}\label{DSFsec}

We consider Linear Time-Invariant (LTI) systems of the following form:
\begin{equation}\label{LTI}
\begin{aligned}
\dot{x} &= Ax + Bu\\
y &= [\begin{array}{cc}I & 0\end{array}]x
\end{aligned}
\end{equation}

\noindent where $x \in \mathbb{R}^n$ is the full state vector, $u \in \mathbb{R}^m$ is the vector of inputs, $y \in \mathbb{R}^p$ is the vector of measured states, with $0<p<n$, and $I$ is the $p\times p$ identity matrix. That is, we assume that some of the states are directly measured and some are not. It is also possible to consider a more general form of $C$ matrix, as in \cite{enoch}. By eliminating the hidden states, the dynamical structure function representation can be derived (see \cite{QP2008}) as:
\begin{equation}
Y = QY + PU
\end{equation}

\noindent where $Y$ and $U$ are the Laplace transforms of $y$ and $u$ and $Q$ and $P$ are strictly proper transfer matrices. The hollow matrix $Q$ is the \emph{Internal Structure} and dictates direct causal relationships between measured states, which may occur via hidden states. The matrix $P$ is the \emph{Control Structure} and similarly defines the relationships between inputs and measured states that are direct in the sense that they do not occur via any other measured states. The dynamical structure function is defined as $(Q,P)$ and the Boolean dynamical structure is defined as the Boolean matrices $\mathcal{B}(Q)$, $\mathcal{B}(P)$ that have the same zero elements as $Q$ and $P$.

\subsection{Dynamical Structure Reconstruction}

The problem of network reconstruction was cast in \cite{QP2008} as a two-stage process, whereby the transfer matrix $G$ is first obtained from input-output data by standard system identification techniques and the dynamical structure function is then obtained from $G$. Here we state the main results for the second stage of this process. The dynamical structure function for a given state-space realisation is unique, and related to the transfer matrix, $G$, as follows:
\begin{equation}\label{consistency}
G = (I-Q)^{-1}P
\end{equation}

\noindent Whilst every $(Q,P)$ uniquely specifies a $G$, there are many such possible $(Q,P)$s for any given $G$ and a dynamical structure function is said to be \emph{consistent} with a given transfer matrix if, and only if, there exists a state-space realisation for which the dynamical structure function (at the considered resolution) satisfies \eqref{consistency}. Exact reconstruction is therefore possible if, and only if, there is only one $(Q,P)$ that is consistent with $G$, which requires some \emph{a priori} knowledge of $(Q,P)$. Corollary \ref{cor1} defines an experimental protocol under which this condition is met.
\begin{mycol}[\cite{QP2008}] \label{cor1}
If $G$ is full rank and $p$ inputs are applied, where there are $p$ measured states and each input directly and uniquely affects only one measured state (such that the matrix $P$ can be made diagonal), then $(Q,P)$ can be recovered exactly from $G$.
\end{mycol}

\noindent Essentially the zero elements of $P$ comprise sufficient knowledge to obtain $(Q,P)$ directly from \eqref{consistency}. Here we assume that the conditions of the above corollary have been met, which ensures solution uniqueness when $G$ is known perfectly.

\subsection{Robust Dynamical Structure Reconstruction}

In practice, the transfer matrix $G$ will be a noisy estimate of the true transfer matrix, as considered in \cite{rQP}. As a result, the dynamical structure function obtained directly from $G$ may not be the best approximation of that of the true system. The noise is expressed as a perturbation on the true transfer matrix $G_{t}$, for example as feedback uncertainty: $G_t = (I+\Lambda)^{-1}G$ for some transfer matrix $\Lambda$. Since $G$ will typically admit a fully-connected dynamical structure function, a strategy to estimate the Boolean dynamical structure consistent with $G_t$ is required.

Consider the $i^{th}$ Boolean dynamical structure function and denote $(Q_i,P_i)$ as a dynamical structure function with this Boolean structure. We can relate $G$ to a transfer matrix $G_i$ that is consistent with $(Q_i,P_i)$ by $G = (I+\Delta_i)G_i$ for some transfer matrix $\Delta_i$. In this case, \eqref{consistency} becomes: $(I+\Delta_i)^{-1}G = (I-Q_i)^{-1}P_i$, which can be rearranged as:
\begin{equation}
\begin{aligned}
\Delta_i &= GP_i^{-1}(I - Q_i) - I \\
&= GX_i - I
\end{aligned}
\end{equation}

\noindent where $X_i = P_i^{-1}(I - Q_i)$ and has the same non-diagonal Boolean structure as $Q_i$. By minimising some norm of $\Delta_i$, with respect to $X_i$, we can obtain the $X_i$ and hence the corresponding dynamical structure function that is consistent with the closest transfer matrix to $G$.

The approach of \cite{rQP} was to minimise $\|\Delta_i\|$ for every possible Boolean $Q_i$ (of which there are $2^{p^2-p}$ since $Q$ has $p^2-p$ degrees of freedom), then use Akaike\rq{}s Information Criterion (AIC) \cite{AIC} to select a solution by penalising the number of nonzero elements in $Q$. Specifically, let $\mathcal{X}_i$ be the set of all $X$ that satisfy the constraints of the $i^{th}$ Boolean $X$, and minimise the Frobenious norm over $s=j\omega$ of $\Delta_i$ as follows:
\begin{equation}\label{mindelta}
\delta_i^2 = \inf_{X \in \mathcal{X}_i} \|GX - I\|_F^2
\end{equation}

\noindent to obtain a measure of the smallest distance $\delta_i$ from $G$ to $G_i$. This choice of norm allows the problem to be cast as a least squares optimisation, and we denote this method $M_0$.

\begin{algorithm}
\caption*{\textbf{Algorithm $M_0$}}
\label{alg0}
\begin{algorithmic}
\For{$i=1 \to 2^{p^2-p}$}
\State $\delta_i^2 = \inf_{X \in \mathcal{X}_i} \|GX - I\|_F^2$
\State $X_i = \arg\inf_{X \in \mathcal{X}_i} \|GX - I\|_F^2$
\EndFor
\State Apply AIC to the set $\{X_i\}$.
\end{algorithmic}
\end{algorithm}

Hence every Boolean structure $\mathcal{B}(X_i)$ can be associated with a distance measure $\delta_i$ from \eqref{mindelta} and a dynamic structure $X_i$, which is the corresponding minimising argument. The principal problem with this approach is that the computational complexity is dominated by the number of optimisations that must be performed, which can be reduced to $p2^p$ by performing the optimisation of the columns of $X$ separately.

\section{Main Result}\label{mainsec}

\subsection{A Polynomial Time Algorithm}\label{ssecpoly}

Here we propose an algorithm with polynomial complexity to estimate the dynamical structure function of $G_t$, under the conditions of Corollary \ref{cor1}. First, an iterative procedure is used to obtain a set $\mathbb{S}$, containing Boolean internal structures, with one structure ($S^j$) for each level of sparsity ($j$ links). Then a model selection procedure is applied to this reduced set to select a single solution. This method is denoted $M_2$ and defined as follows:

\begin{algorithm}
\caption*{\textbf{Algorithm $M_2$}}
\label{alg2}
\begin{algorithmic}
\State Set $S^{p^2-p}$ as the fully-connected structure.
\For{$j = p^2-p \to 1$}
\State Remove one link of $S^{j}$ at a time to obtain a set of $j$ structures with $j-1$ links and calculate $\delta^{j-1}$ for each of these structures.
\State Set $S^{j-1}$ as the minimum-$\delta^{j-1}$ structure.
\EndFor
\State Set $S^0$ as the decoupled structure.
\State Apply a model selection procedure to the set $\mathbb{S} = \{S^j\}$.
\end{algorithmic}
\end{algorithm}

\noindent Fig. \ref{m3_eg} illustrates this procedure for a $p=3$ example. The number of structures that must be considered is exactly $\frac{1}{2}(p^4-2p^3+2p^2-p) + 1$. By optimising the columns of $X$ separately, the overall number of optimisations that must be performed is of the order $O(p^3)$, and since the complexity of the optimisations is also polynomial, the overall computational complexity is polynomial. 

The reasoning behind this approach is that the fully-connected structure is composed of all the links belonging to the true structure plus extra links that afford it a smaller $\delta$ by better modelling the noise. It is intuitive that if the level of noise is not too high, removing false links should have a smaller effect on $\delta$ than removing true links, in which case all the false links would be removed first and the true structure would be encountered. In fact, we will show that if the noise is sufficiently small (in some norm) then this is always the case.

\begin{figure}[t]
\centering
\begin{tikzpicture}
[state/.style={circle,draw, inner sep=0mm, minimum size=2.5mm, font=\small},
link/.style={->,>=latex',semithick},dlink/.style={->,>=latex',dashed,semithick},ddlink/.style={->,>=latex',semithick}]


\node[font=\small]		at (-3,-0.5)		{$S^6$};
\node[state] 	(x1) 		at (-3.5,0)		{$x_1$};
\node[state] 	(x2)		at (-2.5,0) 		{$x_2$};
\node[state]	(x3)		at (-3,0.866)  	{$x_3$};

\draw [link] (x1) to (x2);
\draw [link] (x2) to (x3);
\draw [link] (x3) to (x1);

\draw [dlink] (x3) to [bend left=30] (x2);
\draw [dlink] (x2) to [bend left=30] (x1);
\draw [dlink] (x1) to [bend left=30] (x3);

\node 		at (-2,0.433)		{$\Longrightarrow$};

\node[font=\small]			at (-1,-0.5)		{$S^5$};
\node[state] 	(x1b) 		at (-1.5,0)		{$x_1$};
\node[state] 	(x2b)		at (-0.5,0) 		{$x_2$};
\node[state]	(x3b)		at (-1,0.866)  	{$x_3$};

\draw [link] (x1b) to (x2b);
\draw [link] (x2b) to (x3b);
\draw [link] (x3b) to (x1b);

\draw [dlink] (x2b) to [bend left=30] (x1b);
\draw [dlink] (x1b) to [bend left=30] (x3b);

\node 		at (0,0.433)		{$\Longrightarrow$};

\node[font=\small]			at (1,-0.5)		{$S^4$};
\node[state] 	(x1c) 		at (0.5,0)		{$x_1$};
\node[state] 	(x2c)	 		at (1.5,0) 		{$x_2$};
\node[state]	(x3c)			at (1,0.866)  	{$x_3$};

\draw [link] (x1c) to (x2c);
\draw [link] (x2c) to (x3c);
\draw [link] (x3c) to (x1c);

\draw [dlink] (x1c) to [bend left=30] (x3c);

\node 		at (2,0.433)		{$\Longrightarrow$};

\node[font=\small]			at (3,-0.5)		{$S^3$};
\node[state] 	(x1d) 		at (2.5,0)		{$x_1$};
\node[state] 	(x2d) 		at (3.5,0) 		{$x_2$};
\node[state]	(x3d)		at (3,0.866)  	{$x_3$};

\draw [link] (x1d) to (x2d);
\draw [link] (x2d) to (x3d);
\draw [link] (x3d) to (x1d);

\node 		at (-2,-1.567)		{$\Longrightarrow$};

\node[font=\small]			at (-1,-2.5)		{$S^2$};
\node[state] 	(x1e) 		at (-1.5,-2)		{$x_1$};
\node[state] 	(x2e)		at (-0.5,-2) 		{$x_2$};
\node[state]	(x3e)		at (-1,-1.134)  	{$x_3$};

\draw [link] (x1e) to (x2e);
\draw [link] (x2e) to (x3e);

\node 		at (0,-1.567)		{$\Longrightarrow$};

\node[font=\small]			at (1,-2.5)		{$S^1$};
\node[state] 	(x1f) 		at (0.5,-2)		{$x_1$};
\node[state] 	(x2f)			at (1.5,-2) 		{$x_2$};
\node[state]	(x3f)			at (1,-1.134)  	{$x_3$};

\draw [link] (x1f) to (x2f);

\node 		at (2,-1.567)		{$\Longrightarrow$};

\node[font=\small]			at (3,-2.5)		{$S^0$};
\node[state] 	(x1g) 		at (2.5,-2)		{$x_1$};
\node[state] 	(x2g)	 	at (3.5,-2) 		{$x_2$};
\node[state]	(x3g)		at (3,-1.134)  	{$x_3$};

\end{tikzpicture}
\caption {An example set $\mathbb{S}$ for Method $M_2$ for a network with three measured states. Solid arrows denote true links and dashed arrows denote false links. Starting with the fully-connected structure ($S^6$), one network is obtained for each level of sparsity by successively removing the link which results in the smallest increase in $\delta$. A model selection technique is then used to select a single solution from this set. In this example, the true network is $S^3$.}
\label{m3_eg}
\end{figure}
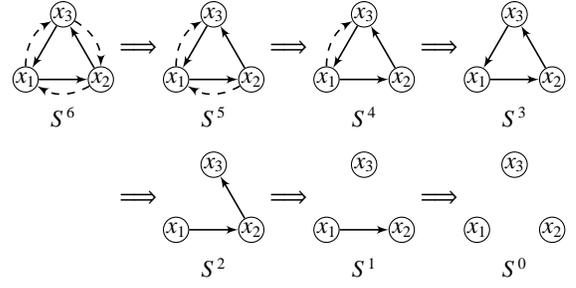

\subsection{An Intermediate Algorithm}
We note that an intermediate algorithm can be defined, with the same two-stage approach of $M_2$ but making use of all $2^{p^2-p}$ Boolean structures and hence with exponential complexity. Rather than obtain the set $\mathbb{S}$ iteratively, this method obtains a set $\mathbb{T}=\{T^j\}$, where each $T^j$ is the minimum-$\delta$ Boolean structure over all Boolean structures with $j$ links.

\begin{algorithm}
\caption*{\textbf{Algorithm $M_1$}}
\label{alg1}
\begin{algorithmic}
\For{$j = 0 \to p^2-p$}
\State Compute $\delta^j$ for all Boolean internal structures with $j$ links.
\State Set $T^j$ as the minimum-$\delta^j$ structure.
\EndFor
\State Apply a model selection procedure to the set $\mathbb{T} = \{T^j\}$.
\end{algorithmic}
\end{algorithm}

Note that the set $\mathbb{T}$ contains the minimum-AIC structures for each level of sparsity, and hence Method $M_0$ is only able to select a solution from this set. If the true Boolean structure is not in the set $\mathbb{T}$, then neither $M_0$ nor $M_1$ can obtain the correct structure, from which we can define:
\begin{mydef}[$M_1$ Solvability]
{A given reconstruction problem is solvable by Method $M_1$ (and $M_0$) if the true Boolean internal structure is in the set $\mathbb{T}$.}
\end{mydef}

\noindent Therefore, if a problem is $M_1$ solvable, Method $M_1$ will always obtain a set $\mathbb{T}$ that contains the true Boolean structure. This step has separated the uncertainty inherent in the problem due to noise from that due to the model selection process. If the model selection stage is not correct but the problem is $M_1$ solvable, then we have a relatively small set of candidate structures, one of which is the true structure.

In summary, $M_0$ considers all $2^{p^2-p}$ Boolean structures and selects a single solution; $M_1$ considers all $2^{p^2-p}$ Boolean structures, selects the best-fitting structure for each number of links to form a subset of $p^2 -p+1$ structures and then selects a single solution from this subset; $M_2$ iteratively finds a set of $p^2-p+1$ Boolean structures and selects a single solution from this set, without having to consider all possible structures. We will next consider a sufficient condition under which a problem can be solved by $M_2$, and it will be seen that this condition is also sufficient for the problem to be $M_1$ solvable.

\subsection{Solvability Conditions for $M_2$}

We define the solvability of $M_2$ as follows:
\begin{mydef}[$M_2$ Solvability]
{A given reconstruction problem is solvable by Method $M_2$ if the true Boolean internal structure is in the set $\mathbb{S}$.}
\end{mydef}

First it is noted that the false links of the fully-connected structure can be removed in any order to obtain the true structure, so the path to the true structure is not unique. A set of \emph{allowable} Boolean structures can be defined as all those that contain at least all of the true links. All other structures with no fewer links than the true structure are \emph{non-allowable}, as to encounter any one of these will mean that the true structure will not be reached. Fig. \ref{allowables} shows $\delta$ values plotted against number of links for an example allowable set plus valid \lq{}paths\rq{} which may be taken between structures by removing one link. The following Lemma provides a sufficient condition for $M_2$ solvability:
\begin{mylem}\label{solvlem}
{A given reconstruction problem is solvable by Method $M_2$ if all allowable structures have smaller $\delta$ than all non-allowable structures with the same number of links.}
\end{mylem}
\begin{proof}
For every allowable structure that is not the true structure, another allowable structure can always be obtained by removing one link, so a path always exists to another allowable structure. If the condition of the Lemma holds, at every stage in the $M_2$ algorithm an allowable structure will be selected in preference to a non-allowable structure until the true structure is reached.
\end{proof}

This is equivalent to all non-allowable structures being within the open shaded region in Fig. \ref{allowables}. Note that if the condition of Lemma \ref{solvlem} holds, this also implies that the problem is $M_1$ solvable. We now derive, for each Boolean structure, an expression for the deviation of the $\delta$ value in the case of noise from its nominal (no noise) value. The level of noise such that the condition of Lemma \ref{solvlem} is met can then be characterised.

From \eqref{mindelta}, the value of $\delta_i$ for the $i^{th}$ Boolean structure is given by:
\begin{equation}\label{mindelta2}
\begin{aligned}
\delta_i^2 &= \inf_{X \in \mathcal{X}_i} \|GX - I\|_F^2 \\
 &= \inf_{Y \in \mathcal{Y}_i} \|AY- b\|_2^2 
\end{aligned}
\end{equation}

\noindent where $A = I \otimes G$, $Y = \mathrm{vec}(X)$ and $b=\mathrm{vec}(I)$ where $I$ is the $p\times p$ identity matrix, $\otimes$ is the Kronecker product and $\mathrm{vec}(\cdot)$ is the vectorization operator. The set $\mathcal{Y}_i$ contains a vector $Y = \mathrm{vec}(X)$ for each element $X \in \mathcal{X}_i$. Consider a subset of Boolean internal structures, which can be obtained from the fully-connected structure by constraining at most one element in each column of $X$ to be zero. The following Lemma relates the $\delta$ values of these structures to that of the fully-connected structure.

\begin{figure}[t]
\centering
\begin{tikzpicture}
[state/.style={circle,draw, inner sep=0mm, minimum size=5mm},
link/.style={->,>=latex',semithick},
dlink/.style={dashed,semithick},
slink/.style={semithick},
xx/.style={coordinate,draw}]


\draw [link] (0,-0.1) to (0,4);
\draw [slink] (-0.1,0) to (0.7,0);
\draw [slink] (0.8,0) to (6,0);

\draw [slink] (0.7,0) to (0.8,0.2);
\draw [slink] (0.7,0) to (0.6,-0.2);
\draw [slink] (0.8,0) to (0.9,0.2);
\draw [slink] (0.8,0) to (0.7,-0.2);

\node at (-0.5,4) {$\delta$};
\node at (3,-1) {Number of Links};

\draw [slink] (6,-0.1) to (6,0.1);
\node [font=\small] at (6,-0.5) {full};
\draw [slink] (4.5,-0.1) to (4.5,0.1);
\node [font=\small] at (4.5,-0.5) {full-1};
\draw [slink] (3,-0.1) to (3,0.1);
\node [font=\small] at (3,-0.5) {full-2};
\draw [slink] (1.5,-0.1) to (1.5,0.1);
\node [font=\small] at (1.5,-0.5) {true};
\node [font=\small] at (0,-0.5) {0};

\node [font=\small] at (-0.5,0) {0};

\fill [color = lightgray] (1.5,3.1) -- (1.5,4) -- (6,4) -- (6,0.5) -- (4.5,1.9) -- (3,2.8) -- (1.5,3.1);

\node [xx] (full) at (6,0.5) {};
\node [font=\small] at (full) {$\times$};
\draw [slink] (-0.1,0.5) to (0.1,0.5);
\node [font=\small] at (-0.5,0.5) {$\delta_{full}$};

\node [xx] (1a) at (4.5,1) {};
\node [font=\small] at (1a) {$\times$};
\node [xx] (1b) at (4.5,1.3) {};
\node [font=\small] at (1b) {$\times$};
\node [xx] (1c) at (4.5,1.9) {};
\node [font=\small] at (1c) {$\times$};

\node [xx] (2a) at (3,1.6) {};
\node [font=\small] at (2a) {$\times$};
\node [xx] (2b) at (3,2.5) {};
\node [font=\small] at (2b) {$\times$};
\node [xx] (2c) at (3,2.8) {};
\node [font=\small] at (2c) {$\times$};

\node [xx] (true) at (1.5,3.1) {};
\node [font=\small] at (true) {$\times$};
\draw [slink] (-0.1,3.1) to (0.1,3.1);
\node [font=\small] at (-0.5,3.1) {$\delta_{true}$};

\draw [slink] (full) to (1a);
\draw [slink] (full) to (1b);
\draw [slink] (full) to (1c);

\draw [slink] (1a) to (2a);
\draw [slink] (1a) to (2b);
\draw [slink] (1b) to (2a);
\draw [slink] (1b) to (2c);
\draw [slink] (1c) to (2b);
\draw [slink] (1c) to (2c);

\draw [slink] (2a) to (true);
\draw [slink] (2b) to (true);
\draw [slink] (2c) to (true);

\end{tikzpicture}
\caption{An example problem showing $\delta$ values plotted against number of links for the allowable set of Method $M_2$. Allowable Boolean structures (those which contain at least all of the true links) are marked by an $\times$. Allowable paths (those that may be taken between allowable structures by removing a single link) are shown as solid lines. The problem is $M_2$ solvable, by Lemma \ref{solvlem}, if all non-allowable structures are within the open shaded region.}
\label{allowables}
\end{figure}
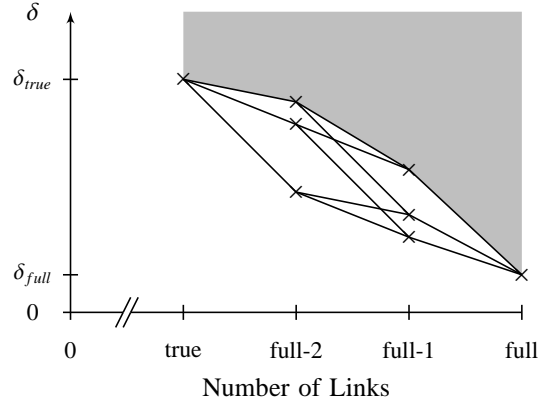

\begin{mylem}\label{optlem}
For any Boolean structure $\mathcal{B}(X_i)$ with no more than one zero element in each column, the value of $\delta_i$ is related to that of the fully-connected structure as follows:
\begin{equation}\label{di_singles}
\delta_i^{2} = \delta_{full}^{2} + \frac{1}{2\pi} \int_{-\infty}^{\infty} \sum_{j}\left(\frac{|Y_{full}(j)|^2}{A^+(j,j)}\right)  \;d\omega
\end{equation}
\noindent for all elements $j$ of $Y$ that are constrained in the $i^{th}$ structure. Here $Y_{full} = \mathrm{vec}(X_{full})$ for the fully-connected structure $X_{full}$, which has $\delta$ value $\delta_{full}$, $A^+ \coloneqq (A\rq{}A)^{-1}$ and $\rq{}$ denotes complex conjugate transpose. 
\end{mylem}

The proof is given in Appendix \ref{appopt} and follows from writing \eqref{mindelta2} as a constrained optimisation problem. In particular, this special case applies to all structures with only one element in total constrained to be zero. We now consider directly how a perturbation on the true transfer matrix $G_t$ affects $\delta_i$ and apply a recursive argument to make use of this special case.

Write the feedback uncertainty in $G_t$ as $\Lambda = \epsilon \Lambda_0$, where $\epsilon = \max_{i,j}\{\|\Lambda(i,j)\|_{\infty}\}$, such that $G = (I+\epsilon\Lambda_0)G_t$. Now given a perturbation $\epsilon\Lambda_0$, parameterise the solution for the $\delta$ value of the $i^{th}$ structure by $\epsilon$ as $\delta_i(\epsilon)$. We first consider the structures obtained by constraining one link of $Y_{full}$.

\begin{mylem}\label{mainlem1}
{Given an estimate of the true transfer matrix $G_t$ as $G = (I+\epsilon\Lambda_0)G_t$, the smallest distance $\delta_i$ from $G_i$ to $G$ for every Boolean structure $\mathcal{B}(X_i)$ with only one zero element is given by:
\begin{equation}
\delta_i^2(\epsilon) = \delta_i^2(0) + \delta_{full}^2(\epsilon) + f_i(\epsilon)
\end{equation}
\noindent where $f_i(\epsilon)$ is continuous and satisfies $f_i(0) = 0$.}
\end{mylem}

The proof is given in Appendix \ref{applem1} and essentially consists of expressing the uncertainty in $\delta_i^2(\epsilon)$ directly using Lemma \ref{optlem}. For any perturbation $\epsilon\Lambda_0$ on $G_t$, the corresponding deviation of $\delta_i^2(\epsilon)$ from its nominal ($\epsilon = 0$) value is therefore given by $\delta_{full}^2(\epsilon) + f_i(\epsilon)$. The following Lemma extends this to describe the deviation of $\delta_i^2(\epsilon)$ for a general $X_i$.

\begin{mylem}\label{mainlem2}
{Given $G = (I+\epsilon\Lambda_0)G_t$, the smallest distance $\delta_i$ from $G_i$ to $G$ for every Boolean structure $\mathcal{B}(X_i)$ is given by:
\begin{equation}
\delta_i^2(\epsilon) = \delta_i^2(0) + \delta_j^2(\epsilon) + f_{ij}(\epsilon)
\end{equation}
\noindent for any Boolean structure $\mathcal{B}(X_j)$ from which $\mathcal{B}(X_i)$ can be obtained by constraining one link. The function $f_{ij}(\epsilon)$ is continuous and satisfies $f_{ij}(0) = 0$.}
\end{mylem}

The proof is given in Appendix \ref{applem2}, where $X_j$ is treated as the solution to a new, unconstrained problem and hence Lemma \ref{mainlem1} can be applied. Using the results of Lemmas \ref{mainlem1} and \ref{mainlem2} it is possible to compute the deviation of every $\delta_i(\epsilon)$ from its nominal $\delta_i(0)$ value. By considering the difference between the deviations of any pair of allowable and non-allowable structures as a function of $\epsilon$, it can be seen that for sufficiently small $\epsilon$, the condition of Lemma \ref{solvlem} is always met and the problem is therefore $M_2$ solvable.

\begin{mytheo}\label{the1}
Given $G = (I+\epsilon\Lambda_0)G_t$, where $G_t$ is the true transfer matrix, there exists an $r > 0$ such that for all $\epsilon < r$, the problem is always solvable by Method $M_2$. Specifically, the difference between $\delta_n$ for any non-allowable structure and $\delta_a$ for any allowable structure with the same number of links is given by:
\begin{equation}
\delta_n^2(\epsilon) - \delta_a^2(\epsilon) = B_n + F_{na}(\epsilon)
\end{equation}
\noindent where the constant $B_n$ satisfies $B_n > 0$ and $|F_{na}(\epsilon)| < B_n$ for all $\epsilon < r$.
\end{mytheo}

The proof is given in Appendix \ref{appthe} and comprises a recursive application of Lemma \ref{mainlem2} to assess the solvability condition of Lemma \ref{solvlem}. This result proves that every problem is $M_2$ solvable for sufficiently small $\epsilon$, and given $G_t$ and $\Lambda_0$ it is possible to calculate a bound $r$ on the size of this $\epsilon$. Method $M_2$ is then guaranteed to encounter the true structure by successively removing links from the fully-connected structure.

\begin{myrem}[Repeated Experiments]
We note that multiple estimates of $G_t$ are readily incorporated by simply replacing $G$ with a block vector of these estimates.
\end{myrem}

\begin{myrem}[Prior Information]
If something is known about the true Boolean structure, this information can easily be taken into account in Algorithm $M_2$ by restricting the number of structures that must be considered. This will then reduce the computational complexity.
\end{myrem}

\begin{myrem}[Steady-State Reconstruction]
As in \cite{rQP}, if only steady-state measurements are available, we can reconstruct the steady-state dynamical structure function from $G(0)$. Complications arise in the case of transfer functions with zero steady-state gain, but otherwise the results of Section \ref{mainsec} can be directly applied to steady-state reconstruction. In particular, the bound of Theorem \ref{the1} is tighter and significantly easier to compute.
\end{myrem}

\section{A New Model Selection Approach}\label{AICsec}

Here we introduce a method of selecting a candidate Boolean structure from the set $\mathbb{S}$ of Method $M_2$, as an alternative to AIC. Our method does not directly penalise model complexity but rather seeks to locate the correct level by identifying the subsequent loss of information as the complexity is further reduced. Denote the $\delta$ value of $S^i \in \mathbb{S}$ as $\delta^i$ and the number of links of the true structure as $i_t$. Note that for no noise, $\delta^i = 0$ for $i_t \leq i \leq p^2-p$ and $\delta^i > 0$ for $i < i_t$, since the solution is unique. Now define the following normalised derivative of $\delta$:
\begin{equation}
d_i^{\prime} = \frac{\delta^{i-1} - \delta^i}{\delta^{i-1}}, \qquad i = 1,\ldots,p^2-p
\end{equation}

\noindent where $d_0^{\prime} \coloneqq 0$ and $ 0 \leq d_i^{\prime} \leq 1$ since $0 \leq \delta^i \leq \delta^{i-1}$. For low levels of noise, $\delta^i$ is close to zero for all $i \geq i_t$ and increases significantly for $i < i_t$, and it is this increase in $\delta^i$ that we seek to detect. In fact, as the noise level approaches zero, $d_i^{\prime} \rightarrow 1 $ for $i=i_t$. Taking a form of second derivative, defined as follows, was found heuristically to improve the distinction of the true structure:
\begin{equation}
d_i^{\prime\prime} = \max \{d_i^{\prime} - d_{i+1}^{\prime}, 0\}, \qquad i = 0,\ldots,p^2-p
\end{equation}

\noindent where $d^{\prime\prime}_{p^2-p+1} \coloneqq 0$ and $0 \leq d_i^{\prime\prime} \leq 1$. The candidate solution is then selected as $S^j$ where $j=\arg\max_i\{d_i^{\prime\prime}\}$ and a measure of the confidence in this selection is given by $\max_i\{d_i^{\prime\prime}\} \in [0,1]$.

\begin{myrem}[A Single Solution]
Whilst it is useful to obtain a single solution for comparative simulations, in practice, a more prudent approach is advised. For example, the relative merit of each of the structures in the solution sets of $M_1$ or $M_2$ could be considered.
\end{myrem}

\section{Simulations}\label{simsec}

Methods $M_1$ and $M_2$ introduced here were compared in simulation with $M_0$ from \cite{rQP} on steady-state network reconstruction of a large number of linear test networks. Networks with three measured states and up to three hidden states were considered; for each of the 64 possible Boolean network structures with three measured states, 300 random, stable linear systems were generated. For each test system and for a range of noise variance from $10^{-5}-10$, three experimental estimates of $G(0)$ were obtained as $G(0) = (I+\Lambda)G_t(0)$, where $G_t$ is the true transfer matrix and the elements of $\Lambda$ were sampled from a zero-mean normal distribution. These estimates of $G(0)$ were then used by each method to attempt to obtain the correct steady-state network structure, given no other information about the true network. In total, $963,900$ network reconstruction problems were considered.

Fig. \ref{fig1} shows the average number of networks that were correctly identified by each method, for each level of noise, as a percentage of the total number of reconstruction problems attempted. Some variation in performance was observed between different network structures (some were harder to identify than others) but the results of Fig. \ref{fig1} are representative of the performance of each of the methods. Also shown in Fig. \ref{fig1} is the $M_1$ solvable limit, which is the percentage of reconstruction problems that were $M_1$ solvable.

The most important result is that the performances of $M_1$ and $M_2$ are almost identical, validating the use of $M_2$ for this problem class. Only $0.37\%$ of problems could be solved by $M_1$ but not by $M_2$, which is certainly justified by the significant reduction in computational complexity. The level of noise required for $M_2$ to fail is apparently similar to that required for $M_1$ to fail. In addition, $M_1$ and $M_2$ (both using the model selection procedure of Section \ref{AICsec}) consistently outperform $M_0$ for all noise levels. For example, for a noise variance of $10^{-3}$, approximately $60\%$ of problems could be solved by $M_0$, whereas $90\%$ could be solved by $M_1$ and $M_2$. 

Figure \ref{fig2} shows the percentage of reconstruction problems for which the set of possible Boolean solutions obtained by each method contained the true Boolean structure. Again the results for methods $M_1$ and $M_2$ are almost identical. For a noise variance of $10^{-2}$, only approximately $30\%$ of problems could be solved by $M_0$, whereas for $M_1$ and $M_2$ we could be $90\%$ certain that the correct Boolean structure is in the set.

\begin{figure}[t]
\centering
\includegraphics[width=0.45\textwidth]{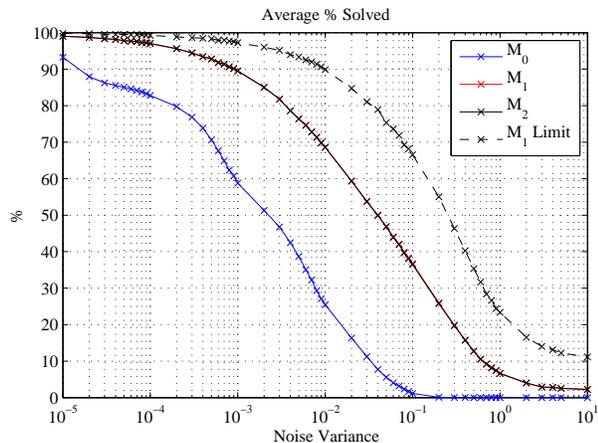}
\caption{Percentage of reconstruction problems successfully solved by each method. $M_0$ is the method from \cite{rQP} with exponential complexity; $M_1$ is an intermediate method also with exponential complexity; $M_2$ is the polynomial complexity method. Also shown is the $M_1$ solvable limit, which is the percentage of reconstruction problems that were solvable by $M_1$ (and $M_0$). The difference between this curve and that of $M_0$ and $M_1$ is due to the incorrect solution being chosen in the model selection stage.}
\label{fig1}
\end{figure}

\section{Conclusions}\label{concsec}
\subsection{Summary}
This paper introduces an algorithm with polynomial complexity that robustly reconstructs the structure and dynamics of an unknown LTI network in the presence of noise and unmodelled nonlinearities. Specifically, we estimate the dynamical structure function from a noisy estimate of the system transfer matrix. Rather than seek to obtain a sparse solution, we consider a set of solutions spanning all levels of sparsity and then select a solution from this set. Following a certain experimental protocol, we prove that every such problem is solvable by our method if the magnitude of the noise is sufficiently small, where the size of this bound depends on the properties of the system in question. The expected performance of this method is assessed in simulation, which demonstrates almost no difference in performance between the exponential and polynomial complexity versions and also shows significant improvements over the previous method.

\subsection{Future Work}
Future work will consider exactly what properties of the true system contribute to the size of the noise bound for that system. This may provide insight to develop a necessary and sufficient condition for solvability, given a certain type of noise perturbation. We also seek to relax the conditions of the experimental protocol, which is another limiting factor in the size of problem that this approach can handle.

\begin{figure}[t]
\centering
\includegraphics[width=0.45\textwidth]{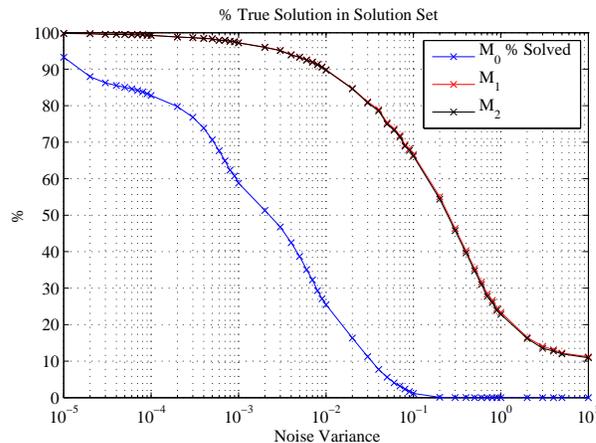}
\caption{Percentage of reconstruction problems where the solution set obtained by each method contained the true solution. $M_0$ is the method from \cite{rQP} with exponential complexity; $M_1$ is an intermediate method also with exponential complexity; $M_2$ is the polynomial complexity method. Note that the solution set of method $M_0$ only contains a single solution.}
\label{fig2}
\end{figure}

\section{Acknowledgements}
This work was supported by EPSRC grants EP/I03210X/1 and EP/P505445/1. 
The authors would also like to thank the reviewers for their helpful comments.

\appendix

\subsection{Proof of Lemma \ref{optlem}} \label{appopt}

\begin{proof}

Equation \eqref{mindelta2} can be written as a constrained optimisation problem as follows:
\begin{equation}
\delta_i^2 = \inf_{Y} \|AY - b\|_2^2 \qquad \textrm{subject to } \quad C_i\rq{}Y = 0
\end{equation}

\noindent where $C_i$ contains the relevant columns of the $p^2 \times p^2$ identity matrix. The constraints are appended to the cost function by a vector of Lagrange multipliers $\mu_i$:
\begin{equation}\label{di_opt}
\begin{aligned}
&\delta_i^2 = \inf_{Y} \frac{1}{2\pi} \int_{-\infty}^{\infty} |AY-b|^2 + \mu_i\rq{}C_i\rq{}Y \; d\omega\\
&C_i\rq{}Y = 0
\end{aligned}
\end{equation}

\noindent with $s=j\omega$ for all transfer functions. Solving for $Y_i$ and $\mu_i$ yields:
\begin{equation}\label{Yi_sol}
\begin{aligned}
&Y_i = (A\rq{}A)^{-1}\left(A\rq{}b - \frac{1}{2}C_i\mu_i\right)\\
&\mu_i = 2\left(C_i\rq{}(A\rq{}A)^{-1}C_i\right)^{-1}C_i\rq{}(A\rq{}A)^{-1}A\rq{}b
\end{aligned}
\end{equation}

\noindent The following abbreviations are made to simplify notation: $A^+ = (A\rq{}A)^{-1}$, $D_i = C_i\rq{}A^+C_i$ and $E_i = C_iD_i^{-1}C_i\rq{}A^+$. It is noted that the fully-connected (unconstrained) solution is $Y_{full} = A^+A\rq{}b$ and the corresponding $\delta_{full}$ given by:
\begin{equation}
\delta_{full}^{2} = \frac{1}{2\pi} \int_{-\infty}^{\infty} b\rq{}\left(I - AA^+A\rq{}\right)b \;d\omega
\end{equation}

\noindent Using \eqref{di_opt} and \eqref{Yi_sol}, for the $i^{th}$ Boolean structure the solution is $Y_i = A^+\left(I-E_i\right)A\rq{}b$ and the optimum $\delta_i$ is given as follows:
\begin{equation}
\begin{aligned}\label{di_sol}
\delta_i^{2} &= \frac{1}{2\pi} \int_{-\infty}^{\infty} b\rq{}\left( I - AA^+A\rq{} + AE_i\rq{}A^+E_iA\rq{} \right)b \;d\omega\\
&= \delta_{full}^{2} + \frac{1}{2\pi} \int_{-\infty}^{\infty} b\rq{}AE_i\rq{}A^+E_iA\rq{}b \;d\omega\\
&= \delta_{full}^{2} + \frac{1}{2\pi} \int_{-\infty}^{\infty} b\rq{}AA^+C_iD_i^{-1}C_i\rq{}A^+A\rq{}b \;d\omega\\
&= \delta_{full}^{2} + \frac{1}{2\pi} \int_{-\infty}^{\infty} Y_{full}\rq{}C_iD_i^{-1}C_i\rq{}Y_{full} \;d\omega
\end{aligned}
\end{equation}

\noindent where $C_i\rq{}Y_{full}$ is the vector of the elements of $Y_{full}$ that are constrained in the $i^{th}$ Boolean structure. Equation \eqref{di_sol} expresses the optimum $\delta$ value of every Boolean structure in terms of the $\delta$ value of the fully-connected structure plus the effect of constraining some of the links to be zero.

The matrix $D_i$ is block-diagonal and composed of elements of $A^+$, which is a block-diagonal matrix given by $A^+ = I \otimes (G\rq{}G)^{-1}$, where $I$ is the $p\times p$ identity matrix. If no more than one element is constrained in each column of $X$, then $D_i$ is diagonal and composed of diagonal elements of $A^+$. In this special case, \eqref{di_sol} reduces to:
\begin{equation}\label{di_singles}
\delta_i^{2} = \delta_{full}^{2} + \frac{1}{2\pi} \int_{-\infty}^{\infty} \sum_{j}\left(\frac{|Y_{full}(j)|^2}{A^+(j,j)}\right)  \;d\omega
\end{equation}

\noindent for all elements $j$ of $Y$ that are constrained in the $i^{th}$ structure.
\end{proof}

\subsection{Proof of Lemma \ref{mainlem1}} \label{applem1}

\begin{proof}
Since $X_i$ has only one zero element, Lemma \ref{optlem} results in:
\begin{equation}\label{di_single}
\delta_{i}^2(\epsilon) - \delta_{full}^2(\epsilon) = \frac{1}{2\pi} \int_{-\infty}^{\infty} \frac{|Y_{full}(j)|^2}{A^+(j,j)}  \;d\omega
\end{equation}

\noindent where $j$ is the index of the constrained element in $Y_{full}$. Since $G = (I+\epsilon\Lambda_0)G_t$, $A$ can be written as $A = (I+(I \otimes \epsilon\Lambda_0))A_t$, where $A_t = I \otimes G_t$. Due to the equivalence of dynamic uncertainty models it is possible to parameterise the uncertainty in a number of forms, which allows us to write: $(A\rq{}A)^{-1} = (A_t\rq{}A_t)^{-1}(I + \Lambda_1(\epsilon))$ for some $\Lambda_1(\epsilon)$ which is also a function of $\Lambda_0$ and $A_t$. Similarly we can write: $(A\rq{}A)^{-1}A\rq{} = (A_t\rq{}A_t)^{-1}A_t\rq{}(I + \Lambda_2(\epsilon))$ for some other $\Lambda_2(\epsilon)$.

Denoting $A_t^+ = (A_t\rq{}A_t)^{-1}$, $\bar{A}_t = A_t^+A_t\rq{}$ and $\bar{A} = A^+A\rq{}$, the integrand of \eqref{di_single} is given by:
\begin{equation}
\begin{aligned}
\frac{|Y_{full}(j)|^2}{A^+(j,j)} &= \frac{\left|\bar{A}(j,:)b\right|^2} {A^+(j,j)} \\
&= \frac{\left|\bar{A}_t(j,:)(I + \Lambda_1)b\right|^2} {A_t^+(j,:) (I(:,j) + \Lambda_2(:,j))} \\
&= \frac{\left|\bar{A}_t(j,:)b\right|^2} {A_t^+(j,j)} \left( 1 - \alpha_j(\epsilon)\right) + \beta_j(\epsilon)
\end{aligned}
\end{equation}
\noindent where $(j,:)$ and $(:,j)$ index respectively the $j^{th}$ row and column of a matrix, and the dependence of $\Lambda_1$ and $\Lambda_2$ on $\epsilon$ has been omitted. The functions $\alpha_j(\epsilon)$ and $\beta_j(\epsilon)$ are continuous functions of $\epsilon$, satisfying $\alpha_j(0) = 0$ and $\beta_j(0) = 0$. Note that $\delta_i^2(0) = \frac{1}{2\pi} \int_{-\infty}^{\infty} \frac{\left|\bar{A}_t(j,:)b\right|^2} {A_t^+(j,j)} \;d\omega$ and hence \eqref{di_single}  can be written:
\begin{equation}\label{di_single2}
\delta_{i}^2(\epsilon) - \delta_{full}^2(\epsilon) = \delta_i^2(0) + f_i(\epsilon)
\end{equation}

\noindent where $f_i(\epsilon) = \frac{1}{2\pi} \int_{-\infty}^{\infty} \beta_j(\epsilon) - \frac{\left|\bar{A}_t(j,:)b\right|^2} {A_t^+(j,j)} \alpha_j(\epsilon)  \;d\omega$ is also a continuous function of $\epsilon$ and satisfies $f_i(0) = 0$.
\end{proof}

\subsection{Proof of Lemma \ref{mainlem2}} \label{applem2}

\begin{proof}
Define $\tilde{Y}_j$ as $Y$ with every element that is constrained to be zero in the $j^{th}$ Boolean structure removed. Similarly $\tilde{A}_j$ is obtained from $A$ by removing the columns of $A$ corresponding to the constrained elements of $Y_j$. Taking $\tilde{Y}_j$ as the unconstrained minimising argument of $\delta_j^2 = \inf_{\tilde{Y}_j} \|\tilde{A}_j\tilde{Y}_j- b\|_2^2$, Lemma \ref{mainlem1} gives:
\begin{equation}\label{di_singleij}
\delta_{i}^2(\epsilon) - \delta_j^2(\epsilon) = \frac{1}{2\pi} \int_{-\infty}^{\infty} \frac{|\tilde{Y}_j(k)|^2}{\tilde{A}_j^+(k,k)}  \;d\omega
\end{equation}

\noindent where $k$ is the index of the constrained element in $\tilde{Y}_j = \tilde{A}_j^+\tilde{A}_j\rq{}b$. Recall that $A_t = I \otimes G_t$ and obtain $\tilde{A}_t$ from $A_t$ by removing the columns of $A_t$ corresponding to the constrained elements of $Y_j$. Then we can write $\tilde{A}_j = (I+(I \otimes \epsilon\Lambda_0))\tilde{A}_t$ and this problem is now in the same form as that of Lemma \ref{mainlem1} and the proof follows as before.
\end{proof}

\subsection{Proof of Theorem \ref{the1}} \label{appthe}

\begin{proof}
A problem is $M_2$ solvable from Lemma \ref{solvlem} if for every allowable structure $Y_a$ and every non-allowable structure $Y_n$ with the same number of links, $\delta_n^2(\epsilon) > \delta_a^2(\epsilon)$. From Lemma \ref{mainlem2}, for every allowable structure $Y_a$, $\delta_a^2(\epsilon)$ is given by:
\begin{equation}
\begin{aligned}
\delta_a^2(\epsilon) &= \delta_{a_1}^2(\epsilon) + f_{a,a_1}(\epsilon) \\
&= \delta_{full}^2(\epsilon) + f_{a,a_1}(\epsilon) + \sum_i f_{a_i,a_{i+1}}(\epsilon) \\
&= \delta_{full}^2(\epsilon) + F_a(\epsilon)
\end{aligned}
\end{equation}

\noindent where $a_i$ denotes an allowable structure with $i$ links more than $Y_a$, from which $Y_a$ can be obtained by removing links. Here we have used the facts that any allowable structure can be obtained by removing one link from another allowable structure and that $\delta_a(0) = 0$ for all allowable structures due to solution uniqueness from Corollary \ref{cor1}. From Lemma \ref{mainlem2}, the functions $f_{a_i,a_{i+1}}(\epsilon)$ are continuous and satisfy $f_{a_i,a_{i+1}}(0) = 0$, hence $F_a(\epsilon)$ (which is a sum of these functions) also satisfies these properties.

Similarly, for every non-allowable structure $Y_n$:
\begin{equation}
\begin{aligned}
\delta_n^2(\epsilon) &= \delta_n^2(0) + \delta_{n_1}^2(\epsilon) + f_{n,n_1}(\epsilon) \\
&= \sum_i \delta_{n_i}^2(0) + \delta_{full}^2(\epsilon) + \sum_i f_{n_i,n_{i+1}}(\epsilon) \\
&= B_n + \delta_{full}^2(\epsilon) + F_n(\epsilon)
\end{aligned}
\end{equation}

\noindent where $n_i$ denotes a non-allowable structure with $i$ links more than $Y_n$, from which $Y_n$ can be obtained by removing links. The constant $B_n = \sum_i \delta_{n_i}^2(0) > 0$ again due to solution uniqueness from Corollary \ref{cor1} and, as in the allowable case, $F_n(\epsilon)$ is continuous and satisfies $F_n(0) = 0$. The solvability condition is then:
\begin{equation}\label{na_solvability}
\delta_n^2(\epsilon) - \delta_a^2(\epsilon) = B_n + F_n(\epsilon) - F_a(\epsilon) > 0
\end{equation}

\noindent where the difference between $\delta_n$ and $\delta_a$ is characterised by a perturbation from each of their nominal ($\epsilon = 0$) values. From the properties of $F_n$ and $F_a$, the function $F_{na} = F_n - F_a$ is also continuous and satisfies $F_{na}(0) = 0$ and hence there always exists some $r_{na} > 0$ such that $|F_{na}(\epsilon)| < B_n$ for all $\epsilon < r_{na}$. The solvability condition of \eqref{na_solvability} is then satisfied as follows:
\begin{equation}
\begin{aligned}
\delta_n^2(\epsilon) - \delta_a^2(\epsilon) &= B_n + F_{na}(\epsilon) \\
&\geq B_n - |F_{na}(\epsilon)| \\
&> 0
\end{aligned}
\end{equation}

\noindent for all $\epsilon < r_{na}$. Taking $r$ to be the minimum $r_{na}$ over all pairs of allowable and non-allowable structures will therefore ensure the problem is $M_2$ solvable for all $\epsilon < r$.
\qedhere
\end{proof}

\end{document}